\documentclass{amsart}

\usepackage{amssymb}
\usepackage{amsthm}
\usepackage{amsmath}
\usepackage{amsbsy}
\usepackage[all]{xy}
\usepackage{bm}
\usepackage{hyperref}
\usepackage{tikz}
\usepackage{array}
\usepackage{colortbl,xcolor}
\usepackage{ytableau}
\usepackage{hhline}
\allowdisplaybreaks

\usepackage{thmtools}
\usepackage{thm-restate}

\newtheorem{theorem}{Theorem}[section]
\newtheorem{proposition}[theorem]{Proposition}
\newtheorem{corollary}[theorem]{Corollary}

\newtheorem{conjecture}[theorem]{Conjecture}
\newtheorem{lemma}[theorem]{Lemma}
\newtheorem{definition}[theorem]{Definition}
\newtheorem{question}[theorem]{Question}

\DeclareMathOperator{\Stab}{Stab}
\DeclareMathOperator{\lcm}{lcm}

\begin{document}

\title[]{A stronger connection between the Erd\H{o}s-Burgess and Davenport constants} 

\author[]{Noah Kravitz}
\address[]{Grace Hopper College, Yale University, New Haven, CT 06510, USA}
\email{noah.kravitz@yale.edu}
\author[]{Ashwin Sah}
\address[]{Massachusetts Institute of Technology, Cambridge, MA 02139 USA}
\email{asah@mit.edu}

\begin{abstract} 
The Erd\H{o}s-Burgess constant of a semigroup $S$ is the smallest positive integer $k$ such that any sequence over $S$ of length $k$ contains a nonempty subsequence whose elements multiply to an idempotent element of $S$.  In the case where $S$ is the multiplicative semigroup of $\mathbb{Z}/n\mathbb{Z}$, we confirm a conjecture connecting the Erd\H{o}s-Burgess constant of $S$ and the Davenport constant of $(\mathbb{Z}/n\mathbb{Z})^{\times}$ for $n$ with at most two prime factors.  We also discuss the extension of our techniques to other rings.
\end{abstract}
\maketitle

\section{Introduction and main results}\label{sec:introduction}

The Erd\H{o}s-Burgess constant is an invariant which measures how much a semigroup avoids idempotent products.  An element $x$ of a multiplicative semigroup is called idempotent if $x^2=x$.  We offer the following formal definition.

\begin{definition}\label{def:constant}
The Erd\H{o}s-Burgess constant of a multiplicative semigroup $S$ (denoted $I(S)$) is the smallest positive integer $k$ such that any sequence of $k$ (not necessarily distinct) elements of $S$ contains a nonempty subsequence (preserving relative order) whose elements multiply to an idempotent element of $S$.  If no such $k$ exists, we say that $I(S)=\infty$.
\end{definition}

The most interesting cases arise when $S$ is the multiplicative semigroup of a finite commutative ring, in which case we let $I_r(R)$ denote the Erd\H{o}s-Burgess constant of the multiplicative semigroup of $R$.  When $R=\mathbb{Z}/n\mathbb{Z}$, clearly the idempotent elements of $R$ are exactly the elements that are equivalent to $0$ or $1$ modulo each prime power dividing $n$.
\\

The problem of computing these constants originated in a question of Erd\H{o}s: Is it always true that $I(S)\leq |S|$ for a finite semigroup $S$?  In 1969, Burgess \cite{burgess1969problem} answered this question in the affirmative when $S$ is commutative or contains only a single idempotent element.  In 1972, Gillam, Hall, and Williams \cite{gillam1972finite} proved the stronger result that $I(S)\leq |S|-|E|+1$ for all finite semigroups $S$, where $E$ is the set of idempotent elements of $S$.  They also showed that this bound is sharp in the sense that for any positive integers $m<n$, there exists a semigroup $S$ with $|S|=n$, $|E|=m$, and $I(S)=|S|-|E|+1$.
\\

The computation of Erd\H{o}s-Burgess constants is closely related to the the study of zero-sum problems.  (See \cite{caro1996zero, gao2006zero} for an overview of this field.)  For a finite additive abelian group $G$, a typical zero-sum problem asks for the smallest positive integer $k$ such that any sequence of $k$ elements of $G$ contains a nonempty subsequence whose terms sum to $0$ while also fulfilling certain other properties.  The most celebrated result in this area is the Erd\H{o}s-Ginzburg-Ziv Theorem \cite{erdos1961theorem}, published in 1961, which says that in any set of $2n-1$ integers, there are $n$ whose sum is divisible by $n$, whereas the same is not true of all sets of $2n-2$ integers.  Popular zero-sum group invariants include the Erd\H{o}s-Ginzburg-Ziv, Olson, Harborth, and Davenport constants.  The last of these will be the most relevant to our study of Erd\H{o}s-Burgess constants.

\begin{definition}\label{def:davenport}
The Davenport constant of a finite abelian group $G$ (denoted $\mathsf{D}(G)$) is the smallest positive integer $k$ such that any sequence of $k$ elements of $G$ contains a nonempty subsequence whose terms sum to $0$.
\end{definition}

The study of this group invariant traces back to a 1963 paper of Rogers \cite{rogers1963combinatorial} and has appeared more recently in a variety of contexts.  (See, e.g., \cite{alford1994there, alon1993zero, delorme2001some, gao2014upper, wang2016davenport}.)
\\

The connection between the Erd\H{o}s-Burgess and Davenport constants first appeared in a recent paper of Wang \cite{wang2018structure} on maximal sequences over semigroups that avoid idempotent products.  When $S$ is a finite abelian group, for instance, the identity is the only idempotent element, so $I(S)=\mathsf{D}(S)$ trivially.  In two papers in 2018, Hao, Wang, and Zhang \cite{hao2018erd, hao2018modular} studied this connection for the multiplicative semigroups of $\mathbb{Z}/n\mathbb{Z}$ and $\mathbb{F}_q[t]/\mathfrak{a}$, where $\mathfrak{a}$ is an ideal of $\mathbb{F}_q[t]$.  For any integer $n>1$, let $\Omega(n)$ denote the total number of primes in the prime factorization of $n$ (with multiplicity), and let $\omega(n)$ denote the number of distinct primes dividing $n$.  Hao, Wang, and Zhang prove the following theorem.

\begin{theorem}[{\cite[Theorem~1.1]{hao2018modular}}]
For any integer $n>1$, we have
$$I_r(\mathbb{Z}/n \mathbb{Z})\geq \mathsf{D}((\mathbb{Z}/n \mathbb{Z})^{\times})+\Omega(n)-\omega(n).$$
Moreover, equality holds if $n$ is either a prime power or a product of distinct primes.
\label{thm:lower}
\end{theorem}

They also conjecture that this inequality is an equality for all $n>1$.

\begin{conjecture}[{\cite[Conjecture~3.2]{hao2018modular}}]\label{conj:equality}
For any integer $n>1$, we have
$$I_r(\mathbb{Z}/n \mathbb{Z})=\mathsf{D}((\mathbb{Z}/n \mathbb{Z})^{\times})+\Omega(n)-\omega(n).$$
\end{conjecture}
In \cite{hao2018erd}, they derive analogous results relating $I_r(\mathbb{F}_q[t]/\mathfrak{a})$ and $\mathsf{D}((\mathbb{F}_q[t]/\mathfrak{a})^{\times})$ and pose the corresponding conjecture.  Wang \cite{wang2018erd} has investigated other aspects of the Erd\H{o}s-Burgess constant, especially in the context of infinite semigroups.
\\

In this paper, we resolve Conjecture~\ref{conj:equality} for some classes of positive integers and make progress on others. In Section $2$, we derive an upper bound on $I_r(\mathbb{Z}/n \mathbb{Z})$ for the case where $n$ has only a single repeated prime factor.  We let $\phi$ denote Euler's totient function.

\begin{restatable}{theorem}{thmsquarefree}\label{thm:squarefree}
Let $n=sp^k$, where $s>1$ is a squarefree integer, $p$ is a prime not dividing $s$, and $k$ is a positive integer.  Then
$$I_r(\mathbb{Z}/n\mathbb{Z}) \leq \mathsf{D}((\mathbb{Z}/n\mathbb{Z})^\times) + (k - 1)+(\phi(s)-1).$$
\end{restatable}
We remark that this upper bound is $\phi(s) - 1$ greater than the conjectured value of $I_r(\mathbb{Z}/n\mathbb{Z})$. In Section $3$, we relate $I_r(\mathbb{Z}/2m\mathbb{Z})$ to $I_r(\mathbb{Z}/m\mathbb{Z})$ when $m$ is odd.

\begin{restatable}{theorem}{thmtwom}\label{thm:twom}
Let $m>1$ be an odd integer.  Then
$$I_r(\mathbb{Z}/2m\mathbb{Z})=I_r(\mathbb{Z}/m\mathbb{Z}).$$
\end{restatable}
In particular, this implies that if an odd integer $m>1$ satisfies Conjecture~\ref{conj:equality} then so does $2m$. Thus Conjecture~\ref{conj:equality} holds for $n$ twice a prime power, using Theorem~\ref{thm:lower}. In Section $4$, we confirm Conjecture~\ref{conj:equality} for the case where exactly two distinct primes appear in the prime factorization of $n$. This is our main result.

\begin{restatable}{theorem}{thmtwoprimes}\label{thm:twoprimes}
Let $n = p^kq^{\ell}$, where $p$ and $q$ are distinct primes and $k$ and $\ell$ are positive integers. Then $$I_r(\mathbb{Z}/n\mathbb{Z}) = \mathsf{D}((\mathbb{Z}/n\mathbb{Z})^\times) + (k - 1) + (\ell - 1).$$
\end{restatable}
Taken together, the previous two results confirm Conjecture~\ref{conj:equality} for $n = 2p^kq^{\ell}$, where $p$ and $q$ are distinct odd primes.
\\

In Section $5$, we generalize Theorem~\ref{thm:lower} to both unique factorization domains and Dedekind domains, which are the rings with a notion of unique prime factorization of elements and ideals, respectively.  In Section $6$, we make some concluding remarks and pose a few questions for future research.

\section{An Upper Bound When Only One Prime Is Repeated}\label{sec:upper}
Before we prove Theorem~\ref{thm:squarefree}, we choose some notation.
\begin{definition}\label{subset-product}
Given a sequence $S$ over a multiplicative semigroup, let $\prod_{\ge k}(S)$ denote the set of all products of at least $k$ elements of $S$.  In other words, $\prod_{\ge k}(S)$ is the set of elements that appear as the product of the elements of some subsequence $T$ in $S$ of length at least $k$.  By convention, let $1 \in \prod_{\geq 0}(S)$ in all cases.
\end{definition}
The following lemma will be useful in both this and the following sections.
\begin{lemma}\label{lem:k-fold}
Let $S=a_1,\ldots,a_{k+t}$ be a sequence over an abelian group $(G,\times)$ of length $k+t$ for some integers $k>0$ and $t\geq 0$.  Let $P = \prod_{\ge k}(S)$. Then either $1\in P$ or $|P|\ge t + 1$.
\end{lemma}
\begin{proof}
The statement $|P|\geq 1$ is trivially true for $t=0$, so we restrict our attention to the case $t\geq 1$.  Suppose $|P|\le t$. We will show that this implies $1\in P$.  Consider the $t + 1$ products $\prod_{i = 1}^{k + j} a_i$ for $0\le j\le t$. (By definition, these are all in $P$.)  The Pigeonhole Principle tells us that some two of these products are equal, so there exist integers $0\le c< d\le t$ such that $\prod_{i=1}^{k+c}a_i=\prod_{i=1}^{k+d}a_i$ and hence $\prod_{i=k+c+1}^{k+d}a_i=1$.
\\

Now, we re-order the elements of $S$ to obtain the sequence $S'=a'_1,\ldots, a'_{k+t}$, where
$$a'_i=
\begin{cases}
a_{k+c+i}, &1\leq i \leq d-c\\
a_{i-(d-c)}, & d-c+1\leq i \leq k+d\\
a_i, &k+d+1\leq i \leq k+t.
\end{cases}$$
In other words, we have moved the $1$-product subsequence of length $d-c$ to the beginning of our sequence and shifted the displaced elements to the right.  If $d-c\geq k$, then we are done.  Otherwise, we can repeat the process described above, which gives us a new $1$-produce subsequence of length $d'-c'$ in front of the $1$-product subsequence of length $d-c$.  Once again, we are done if $(d-c)+(d'-c')\geq k$ because these elements have product $1$.  Otherwise, we continue iterating this process until our $1$-product subsequences have total length at least $k$, which shows that $1 \in P$, as desired. The process must terminate because the $1$-product prefix of our sequence gets strictly longer at each iteration.
\end{proof}
The following adaptation of the methods of \cite{hao2018erd, hao2018modular} allows us to restrict our attention to sequences that do not contain certain elements.
\begin{lemma}\label{lem:reduction}
Suppose $n=p_1p_2\cdots p_r m$, where the $p_i$'s are distinct primes that do not divide $m$, and let $(\mathbb{Z}/n\mathbb{Z})^{\ast}$ denote the set of elements of $\mathbb{Z}/n\mathbb{Z}$ that are relatively prime to all of $p_1, \ldots, p_r$.  If every sequence of length $t$ over $(\mathbb{Z}/n\mathbb{Z})^{\ast}$ contains a nonempty subsequence whose elements multiply to an idempotent element of $\mathbb{Z}/n\mathbb{Z}$, then every sequence of length $t$ over $\mathbb{Z}/n\mathbb{Z}$ also contains a nonempty subsequence whose elements multiply to an idempotent element of $\mathbb{Z}/n\mathbb{Z}$.
\end{lemma}

\begin{proof}
Assume for the sake of contradiction that there exists a sequence $S=a_1,\ldots,a_t$ over $\mathbb{Z}/n\mathbb{Z}$ such that there is no nonempty subsequence of $S$ whose elements multiply to an idempotent element of $\mathbb{Z}/n\mathbb{Z}$. 
For each $a_j$, let $a'_j$ be the unique element of $\mathbb{Z}/n\mathbb{Z}$ that is equivalent to $1\pmod{p_i}$ if $p_i$ divides $a_j$ and $a_j\pmod{p_i}$ otherwise for each $1\leq i \leq r$ and that is also equivalent to $a_j$ modulo $m$. Such a unique element exists by the Chinese Remainder Theorem. Thus, $S'=a'_1,\ldots,a'_t$ is a sequence of length $t$ over $(\mathbb{Z}/n\mathbb{Z})^{\ast}$.  By assumption, $S'$ contains a nonempty subsequence $T'=a'_{j_1},\ldots,a'_{j_{\ell}}$ such that the idempotent product $a'_{j_1}\cdots a'_{j_{\ell}}$ is equivalent to either $0$ or $1$ modulo each prime power dividing $n$.
\\

Consider the product $a_{j_1}\cdots a_{j_{\ell}}$ (which appears as a subsequence $T$ of $S$).  Since $a_j \equiv a'_j \pmod{m}$ for all $j$, it follows that $a_{j_1}\cdots a_{j_{\ell}}$ is still equivalent to $0$ or $1$ modulo each prime power dividing $m$.  We also know that $a'_{j_1}\cdots a'_{j_{\ell}}$ is equivalent to $1$ modulo each $p_i$.  If no $a_{j_k}$ is divisible by $p_i$, then each $a_{j_k} \equiv a'_{j_k} \pmod{p_i}$, and we can conclude that $a_{j_1}\cdots a_{j_{\ell}}$ is equivalent to $1$ modulo $p_i$.  If any $a_{j_k}$ is divisible by $p_i$, then the product $a_{j_1}\cdots a_{j_{\ell}}$ is equivalent to $0$ modulo $p_i$.  So, in both cases, $a_{j_1}\cdots a_{j_{\ell}}$ is equivalent to $0$ or $1$ modulo each prime power dividing $n$, and in fact $a_{j_1}\cdots a_{j_{\ell}}$ is an idempotent element of $\mathbb{Z}/n\mathbb{Z}$.  This yields the desired contradiction.
\end{proof}
Lemma~\ref{lem:reduction} tells us that if we want to establish some $t$ as an upper bound for $I_r(\mathbb{Z}/n\mathbb{Z})$ (with $n$ as in the lemma), it suffices to show that every sequence of length $t$ over $(\mathbb{Z}/n\mathbb{Z})^{\ast}$ contains a nonempty subsequence whose elements multiply to an idempotent element.  In other words, we don't have to worry about sequences containing elements divisible by any of the $p_i$'s.  (The same is not true for primes that divide $n$ multiple times.)  We now prove Theorem~\ref{thm:squarefree}.
\thmsquarefree*
\begin{proof}
Let $(\mathbb{Z}/n\mathbb{Z})^{\ast}$ denote the set of elements of $\mathbb{Z}/n\mathbb{Z}$ that are relatively prime to $s$.  We will show that any sequence of length $N = \mathsf{D}((\mathbb{Z}/n\mathbb{Z})^{\times})+(k-1)+(\phi(s)-1)$ over $(\mathbb{Z}/n\mathbb{Z})^{\ast}$ contains a nonempty subsequence whose elements multiply to an idempotent element of $\mathbb{Z}/n\mathbb{Z}$.  By Lemma~\ref{lem:reduction}, this will be sufficient to establish the result.
\\

Let $S=a_1,\ldots,a_N$ be a sequence over $(\mathbb{Z}/n\mathbb{Z})^{\ast}$ where, without loss of generality, exactly the first $t$ elements are divisible by $p$.  We note that the remaining $\mathsf{D}((\mathbb{Z}/n\mathbb{Z})^{\times})+(k-1)+(\phi(s)-1)-t$ elements are all units of $\mathbb{Z}/n\mathbb{Z}$.  If $t\leq (k-1)+(\phi(s)-1)$, then $S$ contains at least $\mathsf{D}((\mathbb{Z}/n\mathbb{Z})^{\times})$ units.  By the definition of the Davenport constant, this guarantees the existence of a nonempty subsequence of $S$ whose elements multiply to $1$, which is certainly idempotent.  If $t>(k-1)+(\phi(s)-1)$, then we will find a subsequence of $a_1,\ldots, a_t$ of length at least $k$ whose product is equivalent to $1$ modulo $s$. Such a product is idempotent: it is automatically divisible by $p^k$ because these $a_i$'s are all divisible by $p$. Consider  the sequence $a'_1,\ldots,a'_t$ over $(\mathbb{Z}/s\mathbb{Z})^{\times}$ that is obtained by reducing each $a_i$ modulo $s$.  Lemma~\ref{lem:reduction} tells us that either $1 \in \prod_{\geq k}(a'_1,\ldots,a'_t)$ (in which case we are done) or $\left|\prod_{\geq k}(a'_1,\ldots,a'_t)\right|\geq t-k+1\geq \phi(s)$.  In the latter case, $\prod_{\geq k}(a'_1,\ldots,a'_t)$ is the entire group $(\mathbb{Z}/s\mathbb{Z})^{\times}$ (since $|(\mathbb{Z}/s\mathbb{Z})^{\times}|=\phi(s))$, and hence $1 \in \prod_{\geq k}(a'_1,\ldots,a'_t)$.  So, in all cases, $S$ contains a nonempty subsequence whose product is an idempotent element of $\mathbb{Z}/n\mathbb{Z}$, and we can conclude that $\mathsf{D}((\mathbb{Z}/n\mathbb{Z})^{\times})+(k-1)+(\phi(s)-1)$ is in fact an upper bound for $I_r(\mathbb{Z}/n\mathbb{Z})$.
\end{proof}
As mentioned in Section~\ref{sec:introduction}, the upper bound in this lemma is $\phi(s)-1$ greater than the conjectured actual value of $I_r(\mathbb{Z}/n\mathbb{Z})$.

\section{The Case $n=2m$ for Odd $m$}\label{sec:twom}
This short section is devoted to proving Theorem~\ref{thm:twom} and discussing its ramifications for Conjecture~\ref{conj:equality}.
\thmtwom*
\begin{proof}
Let $N = I_r(\mathbb{Z}/m\mathbb{Z})$. First, assume for the sake of contradiction that $I_r(\mathbb{Z}/2m\mathbb{Z})<I_r(\mathbb{Z}/m\mathbb{Z})$.  By definition, there exists a sequence $S=a_1,\ldots,a_{N-1}$ over $\mathbb{Z}/m\mathbb{Z}$ of length $I_r(\mathbb{Z}/m\mathbb{Z})-1$ such that there is no nonempty subsequence of $S$ whose elements multiply to an idempotent element of $\mathbb{Z}/m\mathbb{Z}$.  Consider the sequence $S'=a'_1,\ldots,a'_{N-1}$ over $\mathbb{Z}/2m\mathbb{Z}$, where each $a'_i$ is equivalent to $a_i$ modulo $m$ and $0\leq a_i \leq m-1$.  (As usual, these elements exist by the Chinese Remainder Theorem.)  By assumption, $S'$ contains a nonempty subsequence $T'=b'_1,\ldots,b'_{\ell}$ whose product $x'$ is idempotent in $\mathbb{Z}/2m\mathbb{Z}$.  Then $x'$ is equivalent to either $0$ or $1$ modulo each prime power dividing $2m$.  Consider the corresponding subsequence $T=b_1,\ldots,b_{\ell}$ of $S$ with product $x$.  Because each $a'_i \equiv a_i \pmod{m}$, we have $x \equiv x' \pmod{m}$.  Hence, $x$ remains equivalent to either $0$ or $1$ modulo each prime power dividing $m$, which means that $x$ is idempotent in $\mathbb{Z}/m\mathbb{Z}$.  This yields a contradiction, so in fact $I_r(\mathbb{Z}/2m\mathbb{Z}) \geq I_r(\mathbb{Z}/m\mathbb{Z})$.
\\

Second, assume (again for the sake of contradiction) that $I_r(\mathbb{Z}/2m\mathbb{Z})>I_r(\mathbb{Z}/m\mathbb{Z})$.  Then there exists a sequence $S=a_1,\ldots,a_N$ over $\mathbb{Z}/2m\mathbb{Z}$ of length $I_r(\mathbb{Z}/m\mathbb{Z})$ such that there is no nonempty subsequence of $S$ whose elements multiply to an idempotent element of $\mathbb{Z}/2m\mathbb{Z}$.  Consider the sequence $S'=a'_1,\ldots,a'_N$ over $\mathbb{Z}/m\mathbb{Z}$ where each $a'_i$ is equivalent to $a_i$ modulo $m$.  But $S'$ must contain some nonempty subsequence $T'=b'_1,\ldots,b'_{\ell}$ whose product $x'$ is idempotent in $\mathbb{Z}/n \mathbb{Z}$.  By the same reasoning as above, the corresponding subsequence $T=b_1,\ldots,b_{\ell}$ of $S$ with product $x$ satisfies $x \equiv x' \pmod{m}$.  Hence, $x$ remains equivalent to either $0$ or $1$ modulo each prime power dividing $m$, and, furthermore, $x$ is trivially equivalent to either $0$ or $1$ modulo $2$.  This means that $x$ is idempotent in $\mathbb{Z}/2m \mathbb{Z}$, which yields a contradiction.  So we conclude that $I_r(\mathbb{Z}/2m\mathbb{Z})=I_r(\mathbb{Z}/m\mathbb{Z})$.
\end{proof}
The following consequence of this result holds particular interest.
\begin{corollary}\label{cor:twom}
For any odd integer $m>1$, let $c_m$ be the integer such that $I_r(\mathbb{Z}/m \mathbb{Z})=\mathsf{D}((\mathbb{Z}/m \mathbb{Z})^{\times})+\Omega(m)-\omega(m)+c_m$.  Then we also have
$$I_r(\mathbb{Z}/2m \mathbb{Z})=\mathsf{D}((\mathbb{Z}/2m \mathbb{Z})^{\times})+\Omega(2m)-\omega(2m)+c_m.$$
\end{corollary}
\begin{proof}
Note that $$(\mathbb{Z}/2m \mathbb{Z})^{\times} \cong (\mathbb{Z}/2 \mathbb{Z})^{\times} \times (\mathbb{Z}/m \mathbb{Z})^{\times} \cong 1 \times (\mathbb{Z}/m \mathbb{Z})^{\times} \cong (\mathbb{Z}/m \mathbb{Z})^{\times}.$$  Hence, $\mathsf{D}((\mathbb{Z}/2m \mathbb{Z})^{\times})=\mathsf{D}((\mathbb{Z}/m \mathbb{Z})^{\times})$.  It is also clear that $\Omega(2m)-\omega(2m)=\Omega(m)-\omega(m)$ since $m$ is odd.  Combining these two equalities with Theorem~\ref{thm:twom} establishes the result.
\end{proof}
This corollary tells us that whenever an odd integer $m>1$ satisfies Conjecture~\ref{conj:equality} (i.e., $c_m=0$), $2m$ also satisfies Conjecture~\ref{conj:equality}. As such, we can immediately confirm Conjecture~\ref{conj:equality} for $n$ twice a prime power.
\begin{corollary}\label{cor:minor}
Let $n = 2p^k$, where $p$ is an odd prime and $k$ is a positive integer. Then
$$I_r(\mathbb{Z}/n\mathbb{Z}) = \mathsf{D}((\mathbb{Z}/n\mathbb{Z})^\times) + (k - 1).$$
\end{corollary}
\begin{proof}
This follows immediately from Theorem~\ref{thm:lower} and Corollary~\ref{cor:twom}.
\end{proof}

\section{The Cases $n=p^k q^{\ell}$ and $n=2 p^k q^{\ell}$}\label{sec:twoprimes}
In this section, we prove Theorem~\ref{thm:twoprimes} and an immediate corollary for the case $n=2p^k q^{\ell}$. As usual, we begin with some notation.
\begin{definition}\label{def:set-shift}
Given a sequence $S=a_1,\ldots,a_k$ over a multiplicative semigroup and any element $x$ of the semigroup, let $xS$ denote the sequence $a'_1,\ldots,a'_k$ where each $a'_i=xa_i$.  When we speak of the elements of $S$ as a set (respectively, multiset), the set (multiset) $xS$ is defined in the same fashion.
\end{definition}
We require a lemma on the structure of subset products in abelian groups.
\begin{lemma}[Stabilizer Bound]\label{lem:stabilizer-bound}
Let $S=a_1,\ldots,a_{|S|}$ be a sequence of non-identity elements over an abelian group $(G, \times)$, and let $P = \prod_{\ge 0}(S)$.  If the stabilizer subgroup $\Stab_G(P) = \{x\in G: xP = P\}$ contains only the identity, then $|P|\geq |S|+1$.
\end{lemma}
\begin{proof}
Let $P_i=\prod_{\geq 0}(a_1,\ldots,a_i)$ for each $1\leq i \leq |S|$, so that $P_1=\{1, a_1\}$ and $P_{|S|}=P$.  (Note that $|P_1|=2$ since $a_1 \neq 1$.)  Clearly, each $P_i \subseteq P_{i+1}$.  We will show that this containment is proper, which in turn implies that $|P_i|\geq i+1$ for all $i$.\\

Assume for the sake of contradiction that $P_i=P_{i+1}$ for some $1\leq i \leq |S|-1$.  Writing $P_{i+1}=P_i\cup a_{i+1}P_i$, we see that $a_{i+1}P_i\subseteq P_i$.  Since $|P_i|=|a_{i+1}P_i|$, we must have $a_{i+1}P_i=P_i$, i.e., $a_{i+1} \in \Stab_G(P_i)$.  We claim that $\Stab_G(P_j)\subseteq \Stab_G(P_{j+1})$ for all $1\leq j \leq |S|-1$.  To see this, let $x\in \Stab_G(P_j)$.  Then $xP_j=P_j$ and $x(a_{j+1}P_j)=a_{j+1}P_j$, which implies that $xP_{j+1}=x(P_j\cup a_{j+1}P_j)=(xP_j)\cup (xa_{j+1}P_j)=P_j\cup a_{j+1}P_j=P_{j+1}$.  Thus, we have $a_{i+1}\in \Stab_G(P)$, but this contradicts $\Stab_G(P)$ consisting of only the identity.
\end{proof}
We will also use the following result of Olson \cite{olson1969combinatoriali, olson1969combinatorialii}.
\begin{theorem}[{\cite[Theorem~1.1]{olson1969combinatoriali, olson1969combinatorialii}}]\label{thm:olson}
For a finite abelian group $G=C_{n_1}\times \cdots \times C_{n_r}$, where each $n_i$ divides $n_{i+1}$, define $\mathsf{M}(G)=1+\sum_{i=1}^r (n_i-1)$.  Then $\mathsf{D}(G)\geq \mathsf{M}(G)$.  Moreover, equality holds whenever $r\leq 2$ or $|G|$ is a prime power.
\end{theorem}
We specialize to a case that will be useful in the proof of Theorem~\ref{thm:twoprimes}.
\begin{corollary}
\label{cor:davenport}
For any positive integers $a,b \geq 2$, we have
$$\mathsf{D}(C_a \times C_b)=(\gcd(a,b)-1)+(\lcm(a,b)-1)+1.$$
\end{corollary}
\begin{proof}
The corollary follows from noting that $C_a \times C_b \cong C_{\gcd(a,b)} \times C_{\lcm(a,b)}$ and $\gcd(a, b)$ divides $\lcm(a, b)$.
\end{proof}
Finally, we will need the following simple inequality.
\begin{proposition}\label{prop:gcd-ineq}
For any positive integers $a$, $b$, and $c$ such that $b$ divides $c$, we have
$$(\gcd(a,c)+\lcm(a,c))-(\gcd(a,b)+\lcm(a,b))\geq \frac{c}{b}-1.$$
\end{proposition}
\begin{proof}
Note that $\lcm(a,b)$ divides $\lcm(a,c)$.  We treat the cases $\lcm(a,b)=\lcm(a,c)$ and $\lcm(a,b)<\lcm(a,c)$ separately.
\\

If $\lcm(a,b)=\lcm(a,c)$, then $\frac{c}{b}=\frac{\gcd(a,c)}{\gcd(a,b)}$ since $\gcd(x,y)\lcm(x,y)=xy$ for all positive integers $x$ and $y$.  Since $\gcd(a,b)\geq 1$, we find
$$\gcd(a,c)-\gcd(a,b)\geq \frac{\gcd(a,c)-\gcd(a,b)}{\gcd(a,b)}=\frac{c}{b}-1,$$
and combining this with $\lcm(a,b)=\lcm(a,c)$ establishes the desired inequality.
\\

If $\lcm(a,b)<\lcm(a,c)$, then in fact $\lcm(a,b)\leq \frac{\lcm(a,c)}{2}$ because $\lcm(a,b)$ divides $\lcm(a,c)$.  When $b\geq 2$, we get
$$\lcm(a,c)-\lcm(a,b)\geq \frac{\lcm(a,c)}{2}\geq \frac{c}{2}\geq \frac{c}{b}\geq \frac{c}{b}-1,$$
and combining this with $\gcd(a,c)\geq \gcd(a,b)$ establishes the result.  When $b=1$, we get $\gcd(a,b)=1$ and $\lcm(a,b)=a$.  Using $\lcm(a,c)=\frac{ac}{\gcd(a,c)}$, we also have
$$0\leq \gcd(a,c)\left(\frac{a}{\gcd(a,c)}-1\right)\left(\frac{c}{\gcd(a,c)}-1\right)=\lcm(a,c)-a-c+\gcd(a,c).$$
Rearranging gives
$$(\gcd(a,c)+\lcm(a,c))-(1+a)\geq \frac{c}{1}-1,$$
and substituting $1=\gcd(a,b)$ and $a=\lcm(a,b)$ completes this last case.
\end{proof}
We can now prove Theorem~\ref{thm:twoprimes}.
\thmtwoprimes*
\begin{proof}
We already know from Theorem~\ref{thm:lower} that $I_r(\mathbb{Z}/n\mathbb{Z}) \geq \mathsf{D}((\mathbb{Z}/n\mathbb{Z})^\times) + (k - 1) + (\ell - 1)$, so it remains to show only that this lower bound is also an upper bound.  To this end, assume for the sake of contradiction that there exists some sequence $S$ of length $\mathsf{D}((\mathbb{Z}/n\mathbb{Z})^\times) + (k - 1) + (\ell - 1)$ over $\mathbb{Z}/n \mathbb{Z}$ such that $S$ has no nonempty subsequence the product of whose elements is idempotent.  Recall that an element of $\in \mathbb{Z}/n \mathbb{Z}$ is idempotent exactly when it is equivalent to either $0$ or $1$ modulo $p^k$ and modulo $q^{\ell}$.
\\

If $S$ contains at least $k$ elements divisible by $p$ and $\ell$ elements divisible by $q$, then the product of all of the elements of $S$ is idempotent, which yields a contradiction.  So, without loss of generality, we can restrict our attention to the case where $S$ contains at most $\ell-1$ elements divisible by $q$.  As such, $S$ contains at least $\mathsf{D}((\mathbb{Z}/n\mathbb{Z})^\times) + (k - 1)$ elements not divisible by $q$. We restrict our attention to these elements since the elements divisible by $q$ cannot be used in any idempotent product.
\\

If $S$ contains at most $k-1$ elements divisible by $p$, then it contains at least $\mathsf{D}((\mathbb{Z}/n\mathbb{Z})^\times)$ elements that are not divisible by $p$, i.e., that are units of $\mathbb{Z}/n\mathbb{Z}$.  But then, by the definition of the Davenport constant, $S$ contains a nonempty subsequence whose elements multiply to $1$, which is certainly idempotent.  So we can further restrict our attention to the case where $S$ contains $k+t$ elements divisible by $p$, for some $t\geq 0$.
\\

Let $N = \mathsf{D}((\mathbb{Z}/n\mathbb{Z})^\times)$. We know that $S$ contains the disjoint subsequences $A=a_1,\ldots,a_{k+t}$ and $B=b_1,\ldots,b_{N-t-1}$, where all of the $a_i$'s are divisible by $p$ but not by $q$ and all of the $b_i$'s are units of $\mathbb{Z}/n \mathbb{Z}$ (i.e., are divisible by neither $p$ nor $q$).  We will now focus on the residues of the $a_i$'s and $b_i$'s modulo $q^{\ell}$.  Our goal is to show that there exist $x\in Q_1=\prod_{\geq k}(A)$ and $y \in P_1=\prod_{\geq 0}(B)$ such that $xy \equiv 1 \pmod{q^{\ell}}$.  Then the product $xy$ will be idempotent in $\mathbb{Z}/n \mathbb{Z}$ because $xy \equiv 0 \pmod{p^k}$ by construction.
\\

Let $P_2$ be the set of residues modulo $q^{\ell}$ induced by the elements of $P_1$.  Note that $P_2$ is a subset of $G=(\mathbb{Z}/q^{\ell} \mathbb{Z})^{\times}$, and let $H=\Stab_{G}(P_2)$ be the stabilizer of $P_2$ in $(\mathbb{Z}/q^{\ell} \mathbb{Z})^{\times}$.  Furthermore, let $P_3$ be the set of residues in $G/H$ induced by the elements of $P_2$.  Define the sequence $B'=b'_1,\ldots,b'_{N-t-1}$, where each $b'_i$ is the image of $b_i$ in $G/H$ under the quotient map (after passing through an intermediate element in $G$, if one likes).  Note that $P_3=\prod_{\geq 0}(B')$.
\\

In a similar fashion, let $Q_2$ be the set of residues modulo $q^{\ell}$ induced by the elements of $Q_1$, and let $Q_3$ be the set of residues in $G/H$ induced by the elements of $Q_2$.  Also as above, let $A'=a'_1,\ldots,a'_{k+t}$ be the image of $A$ in $G/H$, where $Q_3=\prod_{\geq k}(A')$.  By Lemma~\ref{lem:k-fold}, we know that either $1 \in Q_3$ or $|Q_3| \geq t+1$.
\\

If $1 \in Q_3$, then there exists some $x \in \prod_{\geq k}(A)$ such that the image of $x$ in $G/H$ is the identity, i.e., $x' \in \Stab_G(P_2)$, where $x'$ is the residue of $x$ modulo $q^{\ell}$.  We know that $1 \in P_1$ (from the empty product) and hence $1 \in P_2$.  Because $x'$ stabilizes $P_2$ in $G$, there exists some $y \in P_1$ such that its image $y'$ in $G$ satisfies $x'y'=1$, i.e., $xy \equiv 1 \pmod{q^{\ell}}$.  But then $xy$ is idempotent, as desired.  For the remainder of the proof, we consider the case $|Q_3| \geq t+1$.  
\\

Consider $x\in \Stab_{G/H}(P_3)$ satisfying $xP_3=P_3$.  Lift this equation to $G$ such that $x$ is lifted to $x'$. We see that $$x'P_2\subseteq \bigcup_{y \in P_2}yH=\left(\bigcup_{y \in P_2}y\right) H=P_2 H=P_2$$ implies $x'P_2=P_2$ and $x'\in \Stab{G}(P_2)=H$.  Thus, $x'$ must reduce to the identity in $G/H$, so $\Stab_{G/H}(P_3)=\{1\}$.  Let $g$ be the number of non-identity elements of $B'$.  By applying Lemma~\ref{lem:stabilizer-bound} to these elements of $B'$, we get $|P_3| \geq g+1$.
\\

If $(t+1)+(g+1)>|G/H|$, then the sets $\{x^{-1}: x \in Q_3\}$ and $P_3$ intersect in $G/H$ by the Pigeonhole Principle.  In other words, there exist $x \in Q_1$ and $z \in P_1$ such that the image of $x^{-1}$ in $G/H$ equals the image of $z$ in $G/H$.  Letting $x'$ and $z'$ be the images of $x$ and $z$ in $G$, we see that $(x')^{-1} \in z'H \subseteq P_2$, where the last inclusion follows from the discussion of the previous paragraph.  Hence, there exists some $y \in P_1$ with image $y'$ in $G$ such that $(x')^{-1}=y'$ and $x'y'=1$.  But this means that $xy \equiv 1 \pmod{q^{\ell}}$, in which case we are done.
\\

We now treat the case where $(t+1)+(g+1)\leq |G/H|$.  Recall that when the sequence $B$ is reduced modulo $q^{\ell}$, exactly $g$ elements end up outside $H$.  So the remaining $(\mathsf{D}((\mathbb{Z}/n\mathbb{Z})^\times)-t-1)-g$ elements of $B$ reduce to elements of $H$.  Let $C$ be the subsequence of these elements, in $\mathbb{Z}/n\mathbb{Z}$.  Recall the decomposition $$(\mathbb{Z}/n \mathbb{Z})^{\times} \cong (\mathbb{Z}/p^k \mathbb{Z})^{\times} \times (\mathbb{Z}/q^{\ell} \mathbb{Z})^{\times} \cong C_{p^{k-1}(p-1)} \times C_{q^{\ell-1}(q-1)}.$$  Corollary~\ref{cor:davenport} tells us that $$\mathsf{D}((\mathbb{Z}/n \mathbb{Z})^{\times})=\gcd(p^{k-1}(p-1),q^{\ell-1}(q-1))+\lcm(p^{k-1}(p-1),q^{\ell-1}(q-1))-1.$$  Because they reduce to elements of $H$ modulo $q^{\ell}$, the elements of $C$ must actually be in a subgroup of $(\mathbb{Z}/n \mathbb{Z})^{\times}$ that is isomorphic to $C_{p^{k-1}(p-1)} \times C_{|H|}$. (Note that $H$ is cyclic because it is a subgroup of the cyclic group $(\mathbb{Z}/q^{\ell}\mathbb{Z})^{\times}$.) In the next paragraph, we will show that $|C|\geq \mathsf{D}(C_{p^{k-1}(p-1)} \times C_{|H|})$.  This will imply that there is a nonempty subsequence of $C$ whose elements multiply to the identity, which is, of course, idempotent in $\mathbb{Z}/n \mathbb{Z}$.
\\

Because
$$|C|=\mathsf{D}((\mathbb{Z}/n\mathbb{Z})^\times)-((t+1)+(g+1))+1\geq \mathsf{D}((\mathbb{Z}/n\mathbb{Z})^\times)-\frac{q^{\ell-1}(q-1)}{|H|}+1,$$
it remains only to show that $$\mathsf{D}((\mathbb{Z}/n\mathbb{Z})^\times)-\frac{q^{\ell-1}(q-1)}{|H|}+1\geq \mathsf{D}(C_{p^{k-1}(p-1)} \times C_{|H|}).$$
Corollary \ref{cor:davenport} tells us that $$\mathsf{D}(C_{p^{k-1}(p-1)} \times C_{|H|})=\gcd(p^{k-1}(p-1),|H|)+\lcm(p^{k-1}(p-1),|H|)-1,$$
and an application of Proposition~\ref{prop:gcd-ineq} with $a=p^{k-1}(p-1)$, $b=|H|$, and $c=q^{\ell-1}(q-1)$ establishes the desired inequality.  This completes the proof.
\end{proof}
This theorem also lets us confirm Conjecture~\ref{conj:equality} for the case $n=2p^k q^{\ell}$.
\begin{corollary}\label{cor:main}
Let $n=2p^k q^{\ell}$, where $p$ and $q$ are distinct odd primes and $k$ and $\ell$ are positive integers.  Then
$$I_r(\mathbb{Z}/n\mathbb{Z})=\mathsf{D}((\mathbb{Z}/n\mathbb{Z})^{\times})+(k-1)+(\ell-1).$$
\end{corollary}
\begin{proof}
This corollary follows immediately from Corollary~\ref{cor:twom} and Theorem~\ref{thm:twoprimes}.
\end{proof}

\section{Other Rings}\label{sec:dedekind}

We now turn to a more general discussion of Erd\H{o}s-Burgess constants in rings.  We focus on the rings in which we can define analogs of $\Omega(n)$ and $\omega(n)$: unique factorization domains (UFDs), which have unique prime factorization of elements, and Dedekind domains, which have unique prime factorization of ideals.  We remark that even though UFDs and Dedekind domains are both extensions of principal ideal domains (PIDs), there exist both UFDs that are not Dedekind domains and Dedekind domains that are not UFDs.  We remark also that UFD and PID are equivalent in a Dedekind domain.  Many of the arguments presented in the previous sections still apply in these more general settings, which unify the cases presented in \cite{hao2018erd, hao2018modular}.
\\

In order to apply the techniques of \cite{hao2018erd, hao2018modular} and the previous sections of this paper, we need a more general Chinese Remainder Theorem. The version stated in the standard algebra text of Atiyah and MacDonald \cite{atiyah1969introduction} will suffice.
\begin{proposition}[{\cite[Proposition~1.10]{atiyah1969introduction}}]\label{prop:crt}
If $\{\mathfrak{a}_1, \ldots, \mathfrak{a}_n\}$ is a set of pairwise coprime ideals of a commutative ring $A$ (i.e., $\mathfrak{a}_i + \mathfrak{a}_j = A$ for all $i\not= j$), then the natural projection map $\phi: A\rightarrow\prod_{i = 1}^n A/\mathfrak{a}_i$ is surjective.
\end{proposition}
We now show that the results of \cite{hao2018erd, hao2018modular} mostly generalize to UFDs.  For any element $a$ of a UFD $R$, let $\Omega(a)$ denote the total number of primes in the prime factorization of $a$ (with multiplicity), and let $\omega(a)$ denote the number of distinct primes (up to multiplication by units) in this prime factorization.
\begin{theorem}\label{thm:ufd}
Let $R$ be a UFD, and let $\mathfrak{a} = (a)$ for some $a\in R$ such that $R/\mathfrak{a}$ is a finite ring.  Then
\[I_r(R/\mathfrak{a})\ge\mathsf{D}((R/\mathfrak{a})^\times) + \Omega(a) - \omega(a).\]
Moreover, equality holds whenever $a$ is a prime power. If $R$ is a PID, then equality also holds whenever $a$ is a product of distinct primes, i.e., $a$ is not divisible by the square of any prime.
\end{theorem}
\begin{proof}
We begin with the lower bound.  We remark that the Davenport constant $\mathsf{D}((R/\mathfrak{a})^\times)$ is finite because $R/\mathfrak{a}$ is finite.  Following the example of \cite{hao2018erd, hao2018modular}, we simply construct a sequence $S$ of length $\mathsf{D}((R/\mathfrak{a})^\times) + \Omega(a) - \omega(a)-1$ that does not contain a nonempty subsequence whose elements multiply to an idempotent element of $R/\mathfrak{a}$.  Write $a=\prod_{i=1}^n p_i^{k_i}$ as a product of powers of distinct primes in $R$.  By the definition of the Davenport constant, there exists a sequence $T$ over $(R/\mathfrak{a})^\times$ of length $\mathsf{D}((R/\mathfrak{a})^\times)-1$ that does not contain a nonempty subsequence the product of whose elements is idempotent.  We obtain the sequence $S$ of length $\mathsf{D}((R/\mathfrak{a})^\times) + \Omega(a) - \omega(a)-1$ by augmenting $T$ by $k_i-1$ elements with representative $p_i$  for each $1\leq i \leq n$, and we claim that this $S$ works.  It is clear that any idempotent element of $R/\mathfrak{a}$ must be equivalent to either $0$ or $1$ modulo each prime power dividing $a$, so there cannot be an idempotent product that includes any elements of $S$ that are not in $T$.  But we know that we cannot make an idempotent product using the only the elements of $T$, so we conclude that $S$ does not contain any nonempty subsequence whose elements multiply to an idempotent product.  This establishes the lower bound.
\\

Next, we show that equality holds whenever $a=p^k$ is a prime power.  Let $N = \mathsf{D}((R/\mathfrak{a})^\times)$ and let $S=a_1,\ldots,a_{N+k-1}$ be a sequence over $R/\mathfrak{a}$ of length $\mathsf{D}((R/\mathfrak{a})^{\times})+k-1$. We will show that $S$ contains a nonempty subsequence the product of whose elements is idempotent.  If at least $k$ elements of $S$ are divisible by $p$, then the product of these elements in $R/\mathfrak{a}$ is $0$, which is certainly idempotent.  If fewer than $k$ elements of $S$ are divisible by $p$, then at least $\mathsf{D}((R/\mathfrak{a})^{\times})$ elements of $S$ are in $(R/\mathfrak{a})^{\times}$.
\\

We must justify the assertion that non-divisibility by $p$ is sufficient for an element $x \in R/\mathfrak{a}$ to be a unit.  The quotient $R/(p)$ is an integral domain because $p$ is prime.  Furthermore, $R/(p)$ is finite (because it is a quotient of $R/\mathfrak{a}$) and hence a field.  Since $x \notin (p)$, its image in $R/(p)$ is nonzero and hence a unit, so (in the lift to $R/\mathfrak{a}$) there exist $y,z \in R/\mathfrak{a}$ such that $xy=1+zp$.  Then $xy(1-(zp)+\cdots+(-1)^{k-1}(zp)^{k-1})=(1+zp)(1-(zp)+\cdots+(-1)^{k-1}(zp)^{k-1})=1+(-1)^{k-1}z^kp^k=1$ shows that $x$ is in fact a unit in $R/\mathfrak{a}$.
\\

Now, by the definition of the Davenport constant, some nonempty product of these units is $1$, which is idempotent.  Hence, in both cases, $S$ contains a nonempty subsequence the product of whose elements is idempotent, which shows that the lower bound is also an upper bound.
\\

Finally, we show that equality holds when $R$ is a PID and $a=p_1\cdots p_n$ is a product of distinct primes in $R$, i.e., $a$ is squarefree.  Because any nonzero prime ideal is maximal in a PID, we see that $\{(p_1), \ldots, (p_n)\}$ is a set of pairwise coprime ideals in $R$, so we can use Proposition~\ref{prop:crt} (Generalized Chinese Remainder Theorem).  By the argument of Lemma~\ref{lem:reduction}, we can establish the upper bound by considering only sequences of elements that are not divisible by any of the $p_i$'s, i.e., sequences of units of $R/\mathfrak{a}$. As above, we must justify the claim that any such element $x$ is a unit in $R/\mathfrak{a}$.  Let $x'$ be any lift of $x$ to $R$.  We know that $x'$ has an inverse modulo each ideal $(p_i)$, i.e., for each $1 \leq i \leq r$, there exist $y_i, z_i \in R$ such that $x'y_i=1+z_ip_i$.  By the Generalized Chinese Remainder Theorem, there exists $y \in R/\mathfrak{a}$ such that $xy=1$ in $R/\mathfrak{a}$, as desired.  Now, similar to above, any sequence $S$ over $(R/\mathfrak{a})^{\times}$ of length $\mathsf{D}((R/\mathfrak{a})^{\times})$ contains a nonempty subsequence whose elements multiply to $1$ by the definition of the Davenport constant.  This completes the proof.
\end{proof}
We now prove the analogous result for Dedekind domains.  For any ideal $\mathfrak{a}$ of a Dedekind domain $R$, let $\Omega(\mathfrak{a})$ denote the total number of prime ideals in the prime ideal factorization of $\mathfrak{a}$ (with multiplicity), and let $\omega(\mathfrak{a})$ denote the number of distinct prime ideals in this factorization.
\begin{theorem}\label{thm:dedekind}
Let $R$ be a Dedekind domain and $\mathfrak{a}$ an ideal of $R$ such that $R/\mathfrak{a}$ is a finite ring. Then
\[I_r(R/\mathfrak{a})\ge\mathsf{D}((R/\mathfrak{a})^\times) + \Omega(\mathfrak{a}) - \omega(\mathfrak{a}).\]
Moreover, equality holds if $\mathfrak{a}$ is either a power of a prime ideal or a product of distinct prime ideals.
\end{theorem}
\begin{proof}
Once again, we begin with the lower bound.  Write $\mathfrak{a}=\prod_{i=1}^n \mathfrak{p}_i^{k_i}$ as a product of powers of distinct prime ideals of $R$.  As in the proof of Theorem~\ref{thm:ufd}, let $T$ be a sequence over $R/\mathfrak{a}$ of length $\mathsf{D}((R/\mathfrak{a})^{\times})-1$ that does not contain a nonempty subsequence the product of whose elements is idempotent.  For each $1\leq i \leq n$, note that $\mathfrak{p}_i^{k_i}\subseteq \mathfrak{p}_i^{k_i-1}$ but these two ideals are not equal because Dedekind domains have unique prime factorization of ideals. (We let $\mathfrak{p}_i^0 = R$.) Hence, the inclusion is proper. Since $\mathfrak{p}_i^{k_i-1}$ is generated by products of the form $r_1\cdots r_{k_i-1}$ with each $r_j\in\mathfrak{p}_i$, there exists some $x_i\in\mathfrak{p}_i^{k_i-1}\backslash\mathfrak{p}_i^{k_i}$ of the form $x_i=a_{i,1} \cdots a_{i,k_i-1}$, where each $a_{i,j} \in \mathfrak{p}_i$.  We now obtain a sequence $S$ of length $\mathsf{D}((R/\mathfrak{a})^\times) + \Omega(\mathfrak{a}) - \omega(\mathfrak{a})-1$ by augmenting $T$ by these $a_{i,j}$ elements (or, rather, their images in $R/\mathfrak{a}$, which retain the inclusion and exclusion properties mentioned above).  We require the following two observations for our claim that $S$ does not contain a nonempty subsequence the product of whose elements is idempotent.
\\

First, we can choose the elements $a_{i, 1}, \ldots, a_{i, k_i-1}$ not to be in any other ideal $\mathfrak{p}_j$.  Since nonzero prime ideals are maximal in Dedekind domains, $\mathfrak{p}_i$ and $\mathfrak{p}_j$ are coprime in $R$, i.e., there exist $x \in \mathfrak{p}_i$ and $y \in \mathfrak{p}_j$ such that $x+y=1$.  Moreover, $\mathfrak{p}_i^{k_i}$ and $\mathfrak{p}_j$ are coprime since $x^{k_i} \in \mathfrak{p_i}^{k_i}$ and $1-(1-y)^{k_i} \in \mathfrak{p}_j$ satisfy $(x^{k_i})+(1-(1-y)^{k_i})=1$.  This lets us apply the Generalized Chinese Remainder Theorem to the set of ideals $\{\mathfrak{p}_1, \ldots, \mathfrak{p}_{i-1}, \mathfrak{p}_i^{k_i}, \mathfrak{p}_{i+1},\ldots, \mathfrak{p}_n\}$, and we can guarantee that each $a_{i, \ell}=1$ in the quotient $R/\mathfrak{p}_j$ for all $i\neq j$.  
\\

Second, suppose $x\in R/\mathfrak{a}$ is an idempotent element that is also in the image of some $\mathfrak{p}_i$. We will show that in fact $x$ is in the image of $\mathfrak{p}_i^{k_i}$.  Let $x'$ be the image of $x$ in the (further) quotient $R/\mathfrak{p}_i^{k_i}$.  Since $x^2=x$ in $R/\mathfrak{a}$, we also have $x'(1-x')=0$ in $R/\mathfrak{p}_i^{k_i}$.  We compute $0=x'(1-x')(1+x'+\cdots+(x')^{k_i-1})=x'(1-(x')^{k_i})=x'$, which implies that $x$ is in the image of $\mathfrak{p}_i^{k_i}$, as desired.
\\


The remainder of the argument proceeds as expected.  Assume for the sake of contradiction that there is some nonempty subsequence $U$ of $S$ the product of whose elements (call it $y$) is idempotent.  Because of the construction of $T$ and the fact that the only idempotent unit is $1$, it is clear that $U$ includes some element $x \in \mathfrak{p}_i$ for some $i$ with $k_i \geq 2$.  Hence, $y \in \mathfrak{p}_i$.  As shown in the previous paragraph, this implies that $y\in \mathfrak{p}_i^{k_i}$ and, moreover, the product $\pi$ of all of the elements of $S$ is also in $\mathfrak{p}_i^{k_i}$.  Since ideal containment in Dedekind domains corresponds to ideal divisibility, $(\pi)\subseteq \mathfrak{p}_i^{k_i}$ implies that there are at least $k_i$ factors of $\mathfrak{p}_i$ in the prime factorization of $(\pi)$.  However, the only elements of $S$ that generate ideals divisible by $\mathfrak{p}_i$ are $a_{i,1},\ldots, a_{i,k_i-1}$, and their product is not in $\mathfrak{p}_i^{k_i}$.  This yields the required contradiction.
\\

When $\mathfrak{a}=\mathfrak{p}^k$ is a prime power, the Pigeonhole Principle argument from the proof of the corresponding part of Theorem~\ref{thm:ufd} applies with no modifications.
\\

Finally, when $\mathfrak{a}=\mathfrak{p}_1\cdots \mathfrak{p}_n$ is a product of distinct prime ideals, the corresponding argument from the proof of Theorem~\ref{thm:ufd} works here, too, because all we needed was the Generalized Chinese Remainder Theorem.

\end{proof}

\section{Concluding Remarks and Open Problems}\label{sec:conclusion}

In this paper, we have confirmed Conjecture~\ref{conj:equality} for many positive integers $n$.  In particular, the conjecture is now known to hold in the following cases:
\begin{itemize}
    \item $n$ is a product of distinct primes (\cite[Theorem 1.1]{hao2018modular}).
    \item $n$ is a prime power (\cite[Theorem 1.1]{hao2018modular}).
    \item $n$ is twice a prime power (Corollary~\ref{cor:minor}).
    \item $n$ has exactly two distinct prime divisors (Theorem~\ref{thm:twoprimes}).
    \item $n$ is double the product of two odd prime powers (Corollary~\ref{cor:main}).
\end{itemize}
We wish to emphasize that the general conjecture for all integers $n>1$ is still open and seems quite difficult.  We consider the following cases particularly approachable for future research:
\begin{itemize}
    \item $n$ has exactly three distinct prime factors.
    \item $n$ is the product of a squarefree integer and a prime power (as discussed in Section $2$).
\end{itemize}
One might also investigate extension results in the style of Theorem~\ref{thm:twom}---for instance, if some $m$ not divisible by $3$ satisfies Conjecture~\ref{conj:equality}, is it always true that $3m$ also satisfies Conjecture~\ref{conj:equality}?
\\

Our proofs of upper bounds in the previous sections suggest a structure result about the ``most difficult'' sequences.  Write $n=p_1^{k_1}\cdots p_r^{k_r}$ as a product of powers of distinct primes.  If we want a product $x$ that is equivalent to either $0$ or $1$ modulo each prime power, then factors of $p_i$ are ``useful'' only when $x$ has at least $k_i$ such factors.  For this reason, it is strictly harder to find an idempotent product when the elements of our sequence $S$ over $\mathbb{Z}/n\mathbb{Z}$ are squarefree with respect to the $p_i$'s, and, in fact, we can consider only sequences of such quasi-squarefree elements in our proofs of upper bounds.  This property could be of use for future computational and experimental work on Erd\H{o}s-Burgess constants.
\\

The inverse Erd\H{o}s-Burgess problem is also of interest: given some integer $n>1$, characterize all sequences $S$ over $\mathbb{Z}/n\mathbb{Z}$ of length $I_r(\mathbb{Z}/n\mathbb{Z})-1$ for which no nonempty subsequence has an idempotent product.  In light of Lemma~\ref{lem:reduction} and the discussion in the previous paragraph, we present the following question.
\begin{question}\label{ques:inverse}
Fix any $n>1$, and write $n=p_1^{k_1}\cdots p_r^{k_r}$ as a product of powers of distinct primes.  Let $S$ be a sequence over $\mathbb{Z}/n\mathbb{Z}$ of length $I_r(\mathbb{Z}/n\mathbb{Z})-1$ that does not have the Erd\H{o}s-Burgess property. Is it true all elements of $S$ are squarefree with respect to each $p_i$ and relatively prime to each $p_i$ for which $k_i=1$?  How else can we characterize the structure of $S$?
\end{question}
For the sake of completeness, we must mention some irregularities in the values of the Davenport constant.  The proof of Theorem~\ref{thm:twoprimes} depends on explicit evaluations of Davenport constants, namely, $\mathsf{D}(G)=\mathsf{M}(G)$ for the relevant rank-$2$ groups $G$.  Although it is known \cite{geroldinger1992davenport} that $\mathsf{D}(G)=\mathsf{M}(G)$ for a few classes of abelian groups beyond what we mention in Theorem~\ref{thm:olson}, it is also known that that this formula fails for infinitely many abelian groups of rank at least $4$. Hence, an approach that uses explicit values of the Davenport constant seems to fail in general but may work when $n$ has three prime factors since the problem of determining the Davenport constant for all rank-$3$ groups remains open. If Conjecture~\ref{conj:equality} turns out to be false, it may be possible to construct counterexamples using these anomalous Davenport constants.
\\

Finally, it would be interesting to see how the results of Sections~\ref{sec:upper} through \ref{sec:twoprimes} generalize to UFDs and Dedekind domains.

\section*{Acknowledgements}
This research was conducted at the University of Minnesota, Duluth REU. It was supported by NSF/DMS grant 1650947 and NSA grant H98230-18-1-0010. The authors would like to thank Joe Gallian for running the program. The authors also thank Joe Gallian and Brice Huang for helpful comments on the manuscript.

\bibliographystyle{plain}
\bibliography{main}

\begin{thebibliography}{10}

\bibitem{alford1994there}
W.~R. Alford, A.~Granville, and C.~Pomerance.
\newblock There are infinitely many {Carmichael} numbers.
\newblock {\em Ann. Math.}, 139(3):703--722, 1994.

\bibitem{alon1993zero}
N.~Alon and M.~Dubiner.
\newblock Zero-sum sets of prescribed size.
\newblock {\em Combinatorics, Paul {Erd\H{o}s} is Eighty}, 1:33--50, 1993.

\bibitem{atiyah1969introduction}
M.~F. Atiyah and I.~G. MacDonald.
\newblock {\em Introduction to Commutative Algebra}.
\newblock Reading, MA: Addison-Wesley, 1969.

\bibitem{burgess1969problem}
D.~A. Burgess.
\newblock A problem on semi-groups.
\newblock {\em Studia Sci. Math. Hungar.}, 4:9--11, 1969.

\bibitem{caro1996zero}
Y.~Caro.
\newblock Zero-sum problems---a survey.
\newblock {\em Discrete Math.}, 152(1-3):93--113, 1996.

\bibitem{delorme2001some}
C.~Delorme, O.~Ordaz, and D.~Quiroz.
\newblock Some remarks on {Davenport} constant.
\newblock {\em Discrete Math.}, 237(1-3):119--128, 2001.

\bibitem{erdos1961theorem}
P.~Erdos, A.~Ginzburg, and A.~Ziv.
\newblock Theorem in the additive number theory.
\newblock {\em Bull. Res. Council Israel F}, 10:41--43, 1961.

\bibitem{gao2006zero}
W.~Gao and A.~Geroldinger.
\newblock Zero-sum problems in finite abelian groups: a survey.
\newblock {\em Expo. Math.}, 24(4):337--369, 2006.

\bibitem{gao2014upper}
W.~Gao, Y.~Li, and J.~Peng.
\newblock An upper bound for the {Davenport} constant of finite groups.
\newblock {\em J. Pure Appl. Algebra}, 218(10):1838--1844, 2014.

\bibitem{geroldinger1992davenport}
A.~Geroldinger and R.~Schneider.
\newblock {On Davenport's constant}.
\newblock {\em J. Combin. Theory Ser. A}, 61(1):147--152, 1992.

\bibitem{gillam1972finite}
D.~W.~H. Gillam, T.~E. Hall, and N.~H. Williams.
\newblock On finite semigroups and idempotents.
\newblock {\em Bull. Lond. Math. Soc.}, 4(2):143--144, 1972.

\bibitem{hao2018erd}
J.~Hao, H.~Wang, and L.~Zhang.
\newblock {Erd\H{o}s-Burgess constant of the multiplicative semigroup of the
  quotient ring of $\mathbb{F}_q[x]$}.
\newblock {\em arXiv:1805.02166v2}, 2018.

\bibitem{hao2018modular}
J.~Hao, H.~Wang, and L.~Zhang.
\newblock {On the modular Erd\H{o}s-Burgess constant}.
\newblock {\em arXiv:1807.07266v1}, 2018.

\bibitem{olson1969combinatoriali}
J.~E. Olson.
\newblock A combinatorial problem on finite abelian groups, {I}.
\newblock {\em J. Number Theory}, 1(1):8--10, 1969.

\bibitem{olson1969combinatorialii}
J.~E. Olson.
\newblock A combinatorial problem on finite abelian groups, {II}.
\newblock {\em J. Number Theory}, 1(2):195--199, 1969.

\bibitem{rogers1963combinatorial}
K.~Rogers.
\newblock A combinatorial problem in {Abelian} groups.
\newblock In {\em Math. Proc. Cambridge Philos. Soc.}, volume~59, pages
  559--562. Cambridge University Press, 1963.

\bibitem{wang2018erd}
G.~Wang.
\newblock {Erd\H{o}s-Burgess constant of the direct product of cyclic
  semigroups}.
\newblock {\em arXiv:1802.08791v1}, 2018.

\bibitem{wang2018structure}
G.~Wang.
\newblock Structure of the largest idempotent-product free sequences in
  semigroups.
\newblock {\em J. Number Theory}, 2018.

\bibitem{wang2016davenport}
G.~Wang and W.~Gao.
\newblock {Davenport constant of the multiplicative semigroup of the ring
  $\mathbb{Z}_{n_1}\oplus\cdots\oplus\mathbb{Z}_{n_r}$}.
\newblock {\em arXiv:1603.06030v1}, 2016.

\end{thebibliography}

\end{document}